\documentclass[11pt]{amsart}
\usepackage[utf8]{inputenc}
\usepackage{amssymb,bbold}
\usepackage{tikz}
\usepackage[left = 1.3in, right = 1.3in, bottom = 1.25in]{geometry}
\usepackage{latexsym,amsmath,amssymb,amscd}
\usepackage{enumerate}
\usepackage{comment}
\newtheorem{theorem}{Theorem}
\newtheorem{lemma}[theorem]{Lemma}
\newtheorem{remark}[theorem]{Remark}
\newtheorem{proposition}[theorem]{Proposition}
\newtheorem{corollary}[theorem]{Corollary}
\theoremstyle{definition}
\newtheorem{definition}{Definition}

\newcommand{\id}{{\rm id}}
\newcommand{\ad}{{\rm ad}}

\newcommand{\Ad}{{\rm Ad}}

\newcommand{\res}{{\rm res}}

\newcommand{\Tr}{{\rm Tr}\,}
\newcommand{\Hc}{{\mathcal H}}
\newcommand{\ug}{{\mathfrak u}}

\begin{document}

\title[The Restricted Schatten-class Grassmannian $\mathrm{Gr}_{\mathrm{res}, p}(\mathcal{H})$ as affine coadjoint orbit]{The Restricted Schatten-class Grassmaniann $\mathrm {Gr}_{\res, p}(\mathcal{H})$ as affine coadjoint orbit}
\author[A.~Tahiri]{Amin Tahiri}
\address{University of Utrecht}
\email{amin-t78@hotmail.com}
\author[A.B.~Tumpach]{Alice Barbora Tumpach}
\address{UMR CNRS 8524\\
UFR de Math\'ematiques\\
Laboratoire Paul Painlev\'e\\
59 655 Villeneuve d'Ascq Cedex\\
France\\ \& \\ Institut CNRS Pauli\\ UMI  CNRS 2842\\ Oskar-Morgenstern-Platz,1 \\1090 Wien\\Austria}
\email{alice-barbora.tumpach@univ-lille.fr}
\thanks{}

\begin{abstract}
In this paper, we consider the restricted $p$-Schatten class Grassmannian $\mathrm {Gr}_{\res, p}(\mathcal{H})$ consisting of infinite-dimensional and infinite codimensional subspaces $W$ of a polarized complex separable Hilbert space $\mathcal{H} = \mathcal{H}_+\oplus \mathcal{H}_-$ such that the orthogonal projection from $W$ onto $\mathcal{H}_+$ is Fredholm and the orthogonal projection from $W$ onto $\mathcal{H}_-$ is in the Schatten ideal $L_p$, $p\geq 1$.  The aim of this paper is to show that, for $1\leq p\leq 2$, the restricted $p$-Schatten class Grassmannian $\mathrm {Gr}_{\res, p}(\mathcal{H})$ is an affine (co-)adjoint orbit of an infinite-dimensional restricted unitary group $\operatorname{U}_{\res, p}(\mathcal{H})$, and that it admits natural weak symplectic structures.  These results follow from the fact that the Lie algebra of the restricted $p$-Schatten class unitary group $\operatorname{U}_{\res, p}(\mathcal{H})$ admits a non-trivial $2$-cocycle.
\end{abstract}

\maketitle
\tableofcontents
\newpage

\section{Introduction}

\subsection{Notation} In this paper, we consider a complex separable Hilbert space $\mathcal{H}$ decomposed into the sum of two infinite-dimensional closed orthogonal subspaces $\mathcal{H}_+$ and $\mathcal{H}_-$, i.e. one has an orthogonal decomposition $\mathcal{H} = \mathcal{H}_+\oplus \mathcal{H}_-$. We will denote by $\mathcal B(\mathcal{H}_1, \mathcal{H}_2)$ the Banach space of bounded linear operators from a Hilbert space $\mathcal{H}_1$ into a Hilbert space $\mathcal{H}_2$, and simply by $\mathcal B(\mathcal{H}_1)$ the Banach space of bounded linear operator from $\mathcal{H}_1$ into itself. 
The Banach Lie group of bounded unitary operators on $\mathcal{H}$ will be denoted by $\operatorname{U}(\mathcal{H}),$ 
\[\operatorname{U}(\mathcal{H})=\{u\in \mathcal B(\mathcal{H})\mid u^*u=uu^*=\id_\mathcal{H}\}, \]
where $\id_\mathcal{H}$ denotes the identity operator on  $\mathcal{H}$.
The Banach Lie algebra of $\operatorname{U}(\mathcal{H})$ consisting of skew-hermitian bounded operators will be denoted by $\mathfrak{u}(\mathcal{H}),$
\[\mathfrak{u}(\mathcal{H})=\{a\in \mathcal B(\mathcal{H})\mid a^*=-a\}.\]

For two Hilbert spaces $\mathcal{H}_1, \mathcal{H}_2$ and $1\leq p< +\infty$, the Schatten class ideal $L_p(\mathcal
{H}_1, \mathcal{H}_2)$ is the Banach space of bounded linear operators $A$ such that
\[
\Tr(A^*A)^{\frac{p}{2}}< +\infty
\]
endowed with the norm $\|A\|_p := \left(\Tr(A^*A)^{\frac{p}{2}}\right)^{\frac{1}{p}}.$ In particular, $L_1(\mathcal{H})$ will denote the Banach space of trace class operators, $L_2(\mathcal{H})$ the Hilbert space of Hilbert-Schmidt operators. Moreover $L_\infty(\mathcal H)$ will denote the Banach space of compact operators endowed with the operator norm. 
For all the Banach Lie algebras used in the  present paper, the Lie bracket is the commutator of operators.

Denote by $\mathrm{pr}_{\pm}\colon \mathcal H\to \mathcal H_\pm$ the orthogonal projections onto $\mathcal{H}_{\pm}$ and define the skew-Hermitian operator  $d = i(\mathrm{pr}_+-\mathrm{pr}_-)$.
For $p\geq 1$, the \textbf{restricted Banach algebra} $L_{\mathrm{res},p}(\mathcal H)$ is the Banach algebra
   \[
        L_{\mathrm{res},p}(\mathcal H)   = \{A\in \mathcal B(\mathcal{H})\mid [d,A]\in L_p(\mathcal{H})\}
        \]
endowed with the norm $\|A\|_{\mathrm{res},p} = \Vert A\Vert_{\infty}
           +\Vert[d,A]\Vert_p$. Similarly, for $q\geq 1$, we define $L_{1,q}(\mathcal H)$ to be the Banach  space
      \begin{align*}
        L_{1,q}(\mathcal H) &  = \{\rho\in \mathcal B(\mathcal{H})\mid [d,\rho]\in L_q(\mathcal{H}), \mathrm{pr}_{\pm}\rho|_{\mathcal{H}_{\pm}}\in L_1(\mathcal{H}_{\pm})\},
      \end{align*}
     equipped with the norm 
     \begin{equation*}
        \|\rho\|_{1,q} = \|\rho_{++}\|_{1} +\|\rho_{--}\|_{1} + \|\rho_{+-}\|_{q} +\|\rho_{-+}\|_{q},    
     \end{equation*}
     where 
     \begin{equation*}
         \rho = \begin{pmatrix}
            \rho_{++} & \rho_{+-}\\
            \rho_{-+} & \rho_{--}
        \end{pmatrix}
     \end{equation*}
     denotes the block decomposition of $\rho$ with respect  to the Hilbert sum $\mathcal{H} = \mathcal{H}_+\oplus \mathcal{H}_-$. The \textbf{restricted trace} (see also \cite{GO-grass}) of $\rho$ is defined as 
\begin{equation*}
    \mathrm{Tr_{res}}(\rho) = \mathrm{Tr}(\rho_{++})+ \mathrm{Tr}(\rho_{--}).
\end{equation*}
Note that the block diagonal operators of $\rho\in L_{1,q}(\mathcal H)$ are trace class.

The restricted unitary group $\mathrm U_{\res, p}(\mathcal{H})$ is defined as follows:
$$
\begin{array}{l}
\operatorname{U}_{\res, p}(\mathcal{H})=\{u\in \operatorname{U}(\mathcal{H})\mid[d,u]\in L_p(\mathcal{H})\}
     =\operatorname{U}(\mathcal{H})\cap L_{\res, p}(\mathcal{H}).
\end{array}
$$
    By the Algebraic Subgroup Theorem \cite[Theorem 1]{harris1977linear}, it admits a natural Banach Lie group structure with Lie algebra 
$$
\mathfrak{u}_{\res, p}=\{a\in\mathfrak{u}(\mathcal{H})\mid[d,a]\in L_p(\mathcal{H})\}
     =\mathfrak{u}(\mathcal{H})\cap L_{\res, p}(\mathcal{H}).
$$
In Section~\ref{sec:duality}, we will show  that a predual to the Banach Lie algebra $\mathfrak{u}_{\res, p}$ is the Banach space:
$$
\mathfrak{u}_{1,q}=\{\rho\in\mathfrak{u}(\mathcal{H}) \mid
[d,\rho]\in L_q(\mathcal{H}),\;
\mathrm{pr}_{\pm}\rho|_{\mathcal{H}_{\pm}}\in L_1(\mathcal{H}_{\pm})\},
$$
with $\frac{1}{p}+\frac{1}{q} = 1$ and $q = 1$ for $p = \infty$.

We will use the notation $\mathfrak{b}^*$ to denote the continuous dual of a Banach space $\mathfrak{b}$, i.e. the Banach space of continuous functionals on $\mathfrak{b}$, and $\mathfrak{b}_*$ for a predual of $\mathfrak{b}$, i.e. for a Banach space such that $(\mathfrak{b}_*)^* = \mathfrak{b}$.

\subsection{Aim of the paper} Set $1< p\leq 2$ and $q$ is defined by $\frac{1}{p} + \frac{1}{q} = 1$ and $q = \infty$ for $p = 1$. The aim of this paper is to show that the restricted $p$--Schatten class Grassmannian can be identified with a family of affine coadjoint orbits $\mathcal{O}_{(0, \gamma)}\subset \mathfrak{u}_{1,q}$ of the unitary group $\operatorname{U}_{\res, p}(\mathcal{H})$ and is naturally endowed with (a family of) symplectic structure(s). The affine action is defined using the non-trivial $2$-cocycle 
\begin{equation}
s(A, B) := \mathrm{Tr_{res}} (A[d, B]),
\end{equation}
$A, B \in \mathfrak{u}_{\res, p}$,
which is the restriction of the Schwinger cocycle to $\mathfrak u_{\mathrm{res},p}\subseteq \mathfrak u_{\mathrm{res},2}$ (cf. \cite{BRT07}). Note that for $1<  p\leq  2$, $\mathrm{Tr_{res}} (d[A, B])$ does not make sense in general  since the diagonal blocks of $[A, B] \in \mathfrak{u}_{\res, p}$ are in general not trace class.

\subsection{Related work} The notion of Banach Poisson manifold was introduced in \cite{{OR03}} and generalized to different contexts in \cite{BGT,beltita-odzijewicz,pelletier,debievre15,NST14,pelletier19,T19}. A comparison of various notions of infinite-dimensional manifolds can be found in \cite{GRT}.  The notion of affine coadjoint orbits can be found in \cite{Nee1}. In \cite{BRT07}, it was proved that the restricted  Grassmannian $\mathrm {Gr}_{\res, 2}(\mathcal{H})$ is an affine coadjoint orbit of the restricted unitary group $\mathrm U_{\res, 2}(\mathcal{H})$.
The notion of Banach Poisson--Lie group was introduced in \cite{T19} and used to shed new light on the relation of the restricted Grassmannian and the Korteweg-de Vries hierarchy. Other completely integrable systems in relation to the restricted Grassmannian were investigated in \cite{GO-grass, GT23, GT24}. Some more formal approaches to infinite dimensional Poisson--Lie groups can be found e.g. in \cite{grabowski94,khesin95,zakharevich94}. The geometry of the unitary groups was studied e.g. in \cite{andruchow10,beltita21,grabowski05}, and it was shown in \cite{GT22} that the unitary group of an Hilbert space is a Banach Poisson--Lie group. For Banach Poisson--Lie groups in duality, the link between the orbits of dressing actions and symplectic leaves is investigated in \cite{TT26}.


\section{The restricted $p$-Schatten class Grassmannian}
In the literature, the restricted Grassmannian refers usually either to the Schatten class Grassmannian $\mathrm{Gr}_{\mathrm{res},2}(\mathcal H)$ modeled on Hilbert-Schmidt operators \cite{segal85}, or to the Grassmannian   $\mathrm{Gr}_{\mathrm{res},\infty}(\mathcal H)$ modeled on compact operators \cite{segal-wilson}. In the present paper, we are mainly interested in the restricted $p$-Schatten class Grassmannian $\mathrm{Gr}_{\mathrm{res},p}(\mathcal H)$ for $1\leq  p \leq 2$ \cite{golinski2025note, stevenson2010}.

 \subsection{Definition of the restricted Grassmannian $\mathrm{Gr}_{\mathrm{res},p}(\mathcal H)$}
\begin{definition}\label{restricted schatten class grassmannian}
     The \textbf{restricted $p$-Schatten class Grassmannian} $\mathrm{Gr}_{\mathrm{res},p}(\mathcal H)$ is defined as the set of all closed subspaces $V\subseteq \mathcal H$ such that 
    \begin{enumerate}
        \item the orthogonal projection $\mathrm{pr}_+\vert_V\colon V\to \mathcal H_+$ is Fredholm (i.e. has finite index);
        \item the orthogonal projection $\mathrm{pr}_-\vert_V\colon V\to \mathcal H_-$ is of class $L_p$.
    \end{enumerate}
\end{definition}
An alternative way of stating Definition \ref{restricted schatten class grassmannian} is as follows. A subspace $V$ belongs to $\mathrm{Gr}_{\mathrm{res},p}(\mathcal H)$ if it is the image of an injective operator $v\colon \mathcal H_+\to \mathcal H$ such $v_+ :=\mathrm{pr}_+\circ v$ is Fredholm and $v_- :=\mathrm{pr}_-\circ v$ is of class $L_p$.

\subsection{$\mathrm{Gr}_{\mathrm{res},p}(\mathcal H)$ as homogeneous space of $\operatorname{U}_{\res, p}(\mathcal{H})$}

We begin with a preliminary lemma that will be used in proof of Proposition \ref{homogeneous}.

\begin{lemma}\label{Fredholm minus identity in L_p}
    Let $A\colon \mathcal H\to \mathcal H$ be Fredholm and $B\in \mathcal B(\mathcal H)$ such that $BA = \mathrm{id}_\mathcal H + P$ with $P\in L_p(\mathcal H)$. Then $AB = \mathrm{id}_{\mathcal H} + Q$ with $Q\in L_p(\mathcal H)$.
\end{lemma}
\begin{proof}
    Since $A$ is Fredholm its range is closed, its kernel is finite dimensional, and its range is of finite codimension (\cite[XI. Theorem 3.1]{conway}). Choose closed complements $\mathcal H= \ker (A)\oplus N$ and $\mathcal H = \mathrm{Ran}(A)\oplus M$ with $\dim(M)<\infty$. Define $T$ to be the inverse of $A\vert_{N}\colon N\to \mathrm{Ran}(A)$ on $\mathrm{Ran}(A)$, and $0$ on $M$. Then $T$ is a bounded operator such that $TA-\mathrm{id}_{\mathcal H}$ and $AT-\mathrm{id}_{\mathcal H}$ are finite-rank operators. 

    Multiplying $BA= \mathrm{id}_{\mathcal H}+ P$ from the right by $T$ we obtain $BAT = (\mathrm{id}_{\mathcal H} + P)T$.
    Since $AT = \mathrm{id}_{\mathcal H} + F$ for $F$ a finite rank operator, this is equal to $B(\mathrm{id}_{\mathcal H} + F) = T+PT$, therefore
    \begin{equation*}
        B = T -BF +PT.
    \end{equation*}
    Multiplying from the left by $A$ yields,
    \begin{equation*}
        AB = AT -ABF +APT.
    \end{equation*}
    Using, again, that $AT = \mathrm{id}_{\mathcal H} + F$, we get 
    \begin{equation*}
        AB  =\mathrm{id}_{\mathcal H} +  F - ABF + APT.
    \end{equation*}
    The operators $F$ and $ABF$ are of finite rank, and $APT$ is of class $L_p$. It follows that $AB - \mathrm{id}_{\mathcal H} \in L_p(\mathcal H)$, as claimed. 
\end{proof}

\begin{proposition}\label{homogeneous}
    For $1< p \leq 2$, the Banach Lie group $\operatorname{U}_{\mathrm{res},p}(\mathcal H)$ acts transitively on the restricted $p$-Schatten class Grassmannian, and the stabilizer of $\mathcal H_+\in \mathrm{Gr}_{\mathrm{res},p}(\mathcal H)$ is $ \operatorname{U}(\mathcal H_+)\times  \operatorname{U}(\mathcal H_-)$.
\end{proposition}
\begin{proof}
 Let $W\in \mathrm{Gr}_{\mathrm{res},p}(\mathcal H)$, we will show that there exists an operator $A\in \mathrm {U}_{\mathrm{res},p}(\mathcal H)$ such that $A(\mathcal H_+) = W$. Let $w\colon \mathcal H_+\to \mathcal H$ be an isometry with image $W$, and $w^\bot\colon \mathcal H_-\to \mathcal H$ an isometry with image $W^\bot$. Then 
    \begin{equation*}
        A:=w\oplus w^\bot\colon \mathcal H_+\oplus \mathcal H_-\to \mathcal H_+\oplus \mathcal H_-
    \end{equation*}
    is a unitary transformation such that $A(\mathcal H_+) = W$. Write 
    \begin{equation*}
        A = \begin{pmatrix}
            w_+&w^\bot_+\\ w_- &w_-^\bot
        \end{pmatrix},
    \end{equation*}
    by definition $w_+$ is Fredholm and $w_-$ is of class $L_p$. Moreover, since $A$ is invertible, it follows that $w_-^\perp$ is Fredholm with $\mathrm{ind}(w^\perp_-)= - \mathrm{ind}(w_+)$. The unitary condition $A^*A = \mathrm{id}_\mathcal H$ implies
    \begin{equation*}
        w^*_+w_+ +w^*_-w_- = \mathrm{id}_{\mathcal H_+},\ \text{ and }\ w^*_+w^\bot_+ +w^*_- w^\bot_- = 0.
    \end{equation*} 
    From Lemma~\ref{Fredholm minus identity in L_p}, it follows that $w_+w^*_+ - \mathrm{id}_{\mathcal H_+} = P$, for an operator $P\in L_p(\mathcal H_+) $. Consequently, by the above equation, 
    \begin{equation*}
        w_+^\perp = - Pw_+^\perp - w_+w_-^*w_-^\perp.
    \end{equation*}
        We conclude that $A\in \mathrm U_{\mathrm{res},p}(\mathcal H)$. The assertion about the stabilizer is obvious.
\end{proof}

\subsection{$\mathrm{Gr}_{\mathrm{res},p}(\mathcal H)$ as a set of orthogonal projectors}
To a closed subspace $V\in \mathrm{Gr}_{\mathrm{res},p}(\mathcal H)$ we associate an orthogonal projection $\mathrm{pr}_V$ onto $V$. 
\begin{proposition}\label{projection grassmannian}
    Let $V\subseteq \mathcal H$  be a closed subspace. Then $V\in \mathrm{Gr}_{\mathrm{res},p}(\mathcal H)$ if and only if $\mathrm{pr}_V- \mathrm{pr}_+ \in L_p$. 
\end{proposition}
\begin{proof}
    Let $V\in \mathrm{Gr}_{\mathrm{res},p}(\mathcal H)$, since the group $\operatorname{U}_{\mathrm{res},p}(\mathcal H)$ acts transitively on $\mathrm{Gr}_{\mathrm{res},p}(\mathcal H)$ we have that $V = u(\mathcal H_+)$ for some $u\in \mathrm U_{\mathrm{res},p}(\mathcal H)$. It follows that 
\begin{equation*}
    \mathrm{pr}_V - \mathrm{pr}_+ = u\mathrm{pr}_+ u^* - \mathrm{pr}_+= [u,\mathrm{pr}_+]u^* \in L_p
\end{equation*}
as the off-diagonal blocks of $u$ are of class $L_p$. For the converse, suppose that $\mathrm{pr}_V- \mathrm{pr}_+\in L_p(\mathcal H)\subseteq L_\infty(\mathcal H)$. Then, by Proposition 5.2.4 in \cite{spectral2023}, the pair $(\mathrm{pr}_V, \mathrm{pr}_+)$ forms a Fredholm pair. It follows that $\mathrm{pr}_+\vert_V$ is Fredholm. Moreover, 
\begin{equation*}
    \mathrm{pr}_-\mathrm{pr}_V = \mathrm{pr}_-(\mathrm{pr}_V- \mathrm{pr}_+)\in L_p(\mathcal H).
\end{equation*}
This concludes the proof.
\end{proof}

\begin{proposition}\label{dual_Grassmannian}
    If $W\in \mathrm{Gr}_{\mathrm{res},p}(\mathcal H)$ then $W^{\perp}$ belongs to the dual Grassmannian $\mathrm{Gr}'_{\mathrm{res},p}(\mathcal H)$ defined as the set of all closed subspaces $W'$ such that
    \begin{enumerate}
        \item the orthogonal projection $\mathrm{pr}_-\vert_{W'}\colon W'\to \mathcal H_-$ is Fredholm (i.e. has finite index);
        \item the orthogonal projection $\mathrm{pr}_+\vert_{W'}\colon W'\to \mathcal H_+$ is of class $L_p$.
    \end{enumerate}
\end{proposition}
\begin{proof}
    Let $W\in \mathrm{Gr}_{\mathrm{res},p}(\mathcal H)$, then, by definition, $\mathrm{pr}_W - \mathrm{pr}_+\in L_p(\mathcal H)$. The orthogonal projection onto $W^\perp$ is given by $\mathrm{pr}_{W^\perp} = \mathrm{id}_{\mathcal H} - \mathrm{pr}_W$. It follows that 
    \begin{equation*}
        \mathrm{pr}_{W^\perp} - \mathrm{pr}_- = (\mathrm{id}_{\mathcal H} - \mathrm{pr}_W) - (\mathrm{id}_{\mathcal H} - \mathrm{pr}_+) = -(\mathrm{pr}_W - \mathrm{pr}_+)\in L_p(\mathcal H),
    \end{equation*}
    which is as desired.
\end{proof}

\subsection{$\mathrm{Gr}_{\mathrm{res},p}(\mathcal H)$ as analytic Banach manifold}

 As the sum of a Fredholm operator with an operator of class $L_p$ is again Fredholm with the same index, we see that if $V$ belongs to $\mathrm{Gr}_{\mathrm{res},p}(\mathcal H)$ then so does the graph of every class $L_p$ operator $V\to V^\bot$. These graphs form the set $\Omega_V$ consisting of all  $W\in \mathrm{Gr}_{\mathrm{res},p}(\mathcal H)$ such that the orthogonal projection $\mathrm{pr}_V\vert_W \colon W\to V$ is an isomorphism. Equivalently, 
\begin{equation*}
    \Omega_V= \{W\in \mathrm{Gr}_{\mathrm{res},p}(\mathcal H)\mid V\oplus W^\bot =\mathcal H\},
\end{equation*}
where $V\oplus W^\perp$ is a topological direct sum.  
These sets are in one-to-one correspondence with the Banach spaces $L_p(V,V^\bot)$ of $L_p$ operators $V\to V^\bot$. More precisely:
\begin{proposition}
    The Schatten class restricted Grassmannian is a Banach manifold modeled on the Banach spaces $L_p(V,V^\bot)$ where ${V\in \mathrm{Gr}_{\mathrm{res},p}(\mathcal H)}$,
\end{proposition}
\begin{proof}
     Let $\Omega_V,\Omega_E\subseteq \mathrm{Gr}_{\mathrm{res},p}(\mathcal H)$ be the open sets associated to  $V,E\in \mathrm{Gr}_{\mathrm{res},p}(\mathcal H)$ defined above. 
    Define the chart associated to the element $V\in \mathrm{Gr}_{\mathrm{res},p}(\mathcal H)$ by 
    \begin{equation*}
        \varphi_V\colon \Omega_V\to L_p(V,V^\bot),\quad W\mapsto  \mathrm{pr}_{V^\bot}\mathrm{pr}_W\mathrm{pr}_V (\mathrm{pr}_V\mathrm{pr}_W\mathrm{pr}_V)^{-1},
    \end{equation*}
    see \cite{golinski2025note, golinski2024Banach}. Note that this map is well-defined as the projection from $W$ onto $V$ is an isomorphism. Moreover, we have that $\mathrm{pr}_W-\mathrm{pr}_V\in L_p$ so 
    \begin{equation*}
        \varphi_V(W) = \mathrm{pr}_{V^\bot}(\mathrm{pr}_W - \mathrm{pr}_V)\mathrm{pr}_V (\mathrm{pr}_V\mathrm{pr}_W\mathrm{pr}_V)^{-1}\in L_p(V,V^\bot).
    \end{equation*}
    We need to show that the set $\varphi_V(\Omega_V\cap \Omega_E)\subseteq L_p(V,V^\bot)$ is open and that the transition map $\psi_{V,E}\colon \varphi_E(\Omega_E\cap \Omega_V)\to \varphi_V(\Omega_V\cap \Omega_E)$ is smooth. 
    Let $E\in \mathrm{Gr}_{\mathrm{res},p}(\mathcal H)$ such that $\Omega_V\cap \Omega_E\neq \emptyset$. The expression of the transition function $\psi_{V,E}$ can be found in \cite{golinski2024Banach} (see also \cite{golinski2025note}) and is given by 
    \begin{equation}\label{transition function}
        \psi_{V,E}(A) = \varphi_V\circ \varphi_E^{-1}(A) = \mathrm{pr}_{V^\bot}(\mathrm{id}_E + A)(\mathrm{pr}_V(\mathrm{pr}_E+A))^{-1},
    \end{equation}
    $ A\in  \varphi_E(\Omega_E\cap \Omega_V)\subseteq L_p(E,E^\bot)$.
    The transition function is clearly well-defined. It remains us to show that $\varphi_E(\Omega_E\cap \Omega_V)\subseteq L_p(V,V^\perp) $ is open and $\psi_{V,E}$ is analytic. The set $\varphi_E(\Omega_E\cap \Omega_V)$ consists out of all operators $A\in L_p(E,E^\bot)$ such that the orthogonal projection from $V$ onto $\mathrm{graph}(A)$ is invertible. Write 
    \begin{equation*}
        \mathrm{id}_\mathcal H = \begin{pmatrix}a& b\\ c&d\end{pmatrix}\colon E\oplus E^\perp \to V\oplus V^\perp,
    \end{equation*}
    then $x+Ax$ ($x\in E$) is sent to $(a+bA)x\in V$, $(c+dA)x\in V^\perp$. So the projection of $W = \mathrm{graph}(A)$ onto $V$ is given by $a+bA$. It follows that $W\in \Omega_V$ if and only if $a+bA$ is invertible. Therefore, 
    \begin{equation*}
        \varphi_E(\Omega_E\cap \Omega_V) = \{A\in L_p(E,E^\bot) \mid a+bA \in \mathrm{GL}(E,V)\},
    \end{equation*}
    which is open inside $L_p(E,E^\perp)$. Note that Equation \eqref{transition function} reduces to 
    \begin{equation*}
        \psi_{V,E}(A) = (c+dA)(a+bA)^{-1}
    \end{equation*}
    which is analytic in $A$. 
    Hence $\mathrm{Gr}_{\mathrm{res},p}(\mathcal H)$ is an analytic Banach manifold modeled on the spaces $L_p(V,V^\bot)$.
\end{proof}


\section{The Banach algebra $\mathfrak{u}_{1,q}(\mathcal{H})$ as the predual of $\mathfrak{u}_{\res, p}(\mathcal{H})$
}\label{sec:duality}

The restricted trace of an operator in ${L}_{1,2}(\mathcal{H})$  was introduced in \cite{GO-grass}, where its trace-like properties were proved. In the present paper, we need to extend its properties with respect to the conjugation of invertible element in $L_{\res, p}(\Hc)$.  

\begin{proposition}\label{restricted trace properties}
    For $A\in L_{\mathrm{res},p}(\mathcal H)$ and $\alpha\in L_{1,q}(\mathcal H)$,  $A\alpha$ and $\alpha A$ are inside $L_{1,q}(\mathcal H)$ and 
    \begin{equation}
        \mathrm{Tr_{res}}(A\alpha) = \mathrm{Tr_{res}}(\alpha A).
    \end{equation}
\end{proposition}

\begin{proof}
    For $A\in  L_{\mathrm{res},p}(\mathcal H)$ and $\alpha\in  L_{1,q}(\mathcal H)$ consider the block decompositions with respect to the Hilbert sum $\mathcal{H} = \mathcal{H}_+\oplus\mathcal{H}_-$ 
    \begin{equation*}
        A =\begin{pmatrix}
            A_{++} & A_{+-}\\
            A_{-+} & A_{--}
        \end{pmatrix},\ \ \text{ and }\  \ \alpha=\begin{pmatrix}
            \alpha_{++} & \alpha_{+-}\\\alpha_{-+} & \alpha_{--}
        \end{pmatrix}.
    \end{equation*}
    Then the operators $A\alpha$ and $\alpha A$ are given by 
    \begin{equation*}
        A\alpha = \begin{pmatrix}
            A_{++}\alpha_{++} + A_{+-}\alpha_{-+} & A_{++}\alpha_{+-}+ A_{+-}\alpha_{--}\\
            A_{-+}\alpha_{++} + A_{--}\alpha_{-+} & A_{-+}\alpha_{+-}+A_{--}\alpha_{--}
        \end{pmatrix}
    \end{equation*}
    and
    \begin{equation*}
        \alpha A = \begin{pmatrix}
            \alpha_{++}A_{++} + \alpha_{+-}A_{-+} & \alpha_{++}A_{+-}+ \alpha_{+-}A_{--}\\
            \alpha_{-+}A_{++} + \alpha_{--}A_{-+} & \alpha_{-+}A_{+-}+\alpha_{--}A_{--}
        \end{pmatrix}.
    \end{equation*}
    By definition,
    \begin{equation*}
        \mathrm{Tr}_{\mathrm{res}}(A\alpha) = \mathrm{Tr} (A_{++}\alpha_{++}+A_{+-}\alpha_{-+}) + \mathrm{Tr}(A_{-+}\alpha_{+-}+ A_{--}\alpha_{--}).
    \end{equation*}
    Since $L_p\cdot L_q\subset L_1$, all terms are trace class and 
    \begin{equation*}
        \mathrm{Tr}_{\mathrm{res}}(A\alpha) = \mathrm{Tr} (A_{++}\alpha_{++})+ \mathrm{Tr} (A_{+-}\alpha_{-+}) + \mathrm{Tr}(A_{-+}\alpha_{+-})+ \mathrm{Tr} (A_{--}\alpha_{--}).
    \end{equation*}
    Since  for $A_{++}\in L_{\infty}(\Hc_+)$ and $\alpha_{++}\in L_1(\Hc_+)$, one has $$\mathrm{Tr} (A_{++}\alpha_{++}) = \mathrm{Tr}(\alpha_{++}A_{++}),$$ and for $A_{+-}\in L_p(\Hc_-, \Hc_+)$ and $\alpha_{-+}\in L_q(\Hc_+, \Hc_-)$, $$\mathrm{Tr}( A_{+-}\alpha_{-+}) = \mathrm{Tr}(\alpha_{-+}A_{+-}),$$   it follows that 
    \begin{equation*}
        \begin{split}
            \mathrm{Tr_{res}}(A\alpha)  =& \mathrm{Tr}(\alpha_{++}A_{++}) + \mathrm{Tr}(\alpha_{-+}A_{+-})\\& +\mathrm{Tr}(\alpha_{+-}A_{-+})+ \mathrm{Tr}(\alpha_{--}A_{--}) = \mathrm{Tr_{res}}(\alpha A),
        \end{split}
    \end{equation*}
     as desired.  
\end{proof}

\begin{proposition}
    For $p> 1$ and $q>1$ such that $\frac{1}{p}+\frac{1}{q} = 1$ as well as for $p = 1$ and $q = \infty$,  the pairing between  $L_{1,q}(\mathcal H)$ and $L_{\mathrm{res},p}(\mathcal H)$ given by 
    \begin{equation*}
        \begin{split}
            \langle \alpha, A\rangle &= \mathrm{Tr_{res}}(\alpha A)\\
            &= \mathrm{Tr}(\alpha_{++}A_{++}) +\mathrm{Tr}(\alpha_{+-}A_{-+}) + \mathrm{Tr}(\alpha_{-+}A_{+-}) + \mathrm{Tr}(\alpha_{--}A_{--}),
        \end{split}
    \end{equation*}
    where $A\in L_{\mathrm{res},p}(\mathcal H)$ and $\alpha\in L_{1,q}(\mathcal H)$, induces an isomorphism $(L_{1,q}(\mathcal H))^* \cong L_{\mathrm{res},p}(\mathcal H)$ of Banach spaces. That is, $L_{1,q}(\mathcal H)$ is predual to $L_{\mathrm{res},p}(\mathcal H)$.
\end{proposition}

\begin{proof}
    This follows directly from the duality between $L_p$ and $L_q$ for $1< p ,q<\infty$, together with the identification $(L_1(\Hc_\pm))^*\cong  \mathcal B(\Hc_\pm)$ and $(L_\infty(\mathcal H_\pm))^*\cong L_1(\mathcal H_\pm)$.
\end{proof}

\begin{proposition}\label{restricted_trace_conjugation_invariance}
Let $g\in L_{\mathrm{res},p}(\mathcal H)$ be invertible. Then for all  $A\in L_{\res, p}(\mathcal H)$ and $\alpha\in L_{1,q}(\mathcal H)$ one has
\begin{equation}\label{conj}
        \langle  \alpha, g A g^{-1}\rangle = \mathrm{Tr_{res}}(\alpha g A g^{-1}) = \mathrm{Tr_{res}}( g^{-1}\alpha g A) = \langle  g^{-1}\alpha g, A\rangle.
    \end{equation}
    In particular,
    \[
\mathrm{Tr_{res}}(g\alpha g^{-1}) = \mathrm{Tr_{res}}(\alpha).
    \]
\end{proposition}
\begin{proof}
    This follows from the Proposition~\ref{restricted trace properties}, since $\alpha g A \in L_{1,q}(\mathcal{H})$ and $g^{-1}\in L_{\mathrm{res},p}(\mathcal H)$, therefore $\mathrm{Tr_{res}}(\alpha g A g^{-1}) = \mathrm{Tr_{res}}(g^{-1}\alpha g A) $. The second equality follows from equation~\eqref{conj} with $A$ equal to the identity.
\end{proof}

\begin{proposition}\label{duality} For $1<p\leq 2$ and $q\geq 2$ such that $\frac{1}{p}+\frac{1}{q} = 1$ as well as for $p = 1$ and $q = \infty$, 
the Banach space $\mathfrak{u}_{1,q}$ is a predual of the unitary
restricted algebra $\mathfrak{u}_{\res,p}$, the duality pairing
$\langle\cdot\,,\cdot\rangle$ being given by the restricted trace
\begin{equation}\label{dualitypairing}
\langle\cdot\,,\cdot\rangle~:\mathfrak{u}_{\res,p}\times \mathfrak{u}_{1,q}\rightarrow{\mathbb
R},\quad (a,\rho)\mapsto\mathrm{Tr_{res}}(a\rho).
\end{equation}
\end{proposition}

\begin{proof}
Consider two arbitrary elements
$$a=\begin{pmatrix} a_{++} & a_{+-} \\
                    -a_{+-}^* & a_{--}
    \end{pmatrix}\in\mathfrak{u}_{\res,p} \quad
\text{ and } \quad
\rho=\begin{pmatrix} \rho_{++} & -\rho_{-+}^* \\
                     \rho_{-+} & \rho_{--}
    \end{pmatrix}\in\mathfrak{u}_{1,q}.$$
Then
\begin{equation}\label{mult}
a\rho=\begin{pmatrix}
       a_{++}\rho_{++}+a_{+-}\rho_{-+} & -a_{++}\rho_{-+}^{*}+a_{+-}\rho_{--} \\
       -a_{+-}^*\rho_{++}+a_{--}\rho_{-+} & a_{+-}^{*}\rho_{-+}^{*}+a_{--}\rho_{--}
    \end{pmatrix}\in L_{1,q}(\Hc),
\end{equation}
hence
\begin{equation}\label{pairing}
\mathrm{Tr_{res}}(a\rho)=\Tr(a_{++}\rho_{++})+2\Re\Tr(a_{+-}\rho_{-+})+
\Tr(a_{--}\rho_{--}),
\end{equation}
where $\Re z$ denotes the real part of the complex number $z $. Recall that the bilinear functional
$$
\mathcal B(\mathcal{H}_{\pm})\times L_1(\mathcal{H}_{\pm})\rightarrow{\mathbb C},\quad
(b,c)\mapsto\Tr(bc),
$$
induces a topological isomorphism of complex Banach spaces
$\left(L_1(\Hc_{\pm})\right)^*\cong \mathcal B(\Hc_{\pm})$. It follows
that the trace induces a topological isomorphism of real Banach
spaces
\begin{equation}\label{unitarypairing}
\left(\mathfrak{u}(\Hc_{\pm})\cap L_1(\Hc_{\pm})\right)^{*}\cong\mathfrak{u}(\Hc_{\pm}).
\end{equation}
Indeed, the $\mathbb{C}$-linearity of the trace implies that for
$b\in \mathcal B(\Hc_{\pm})$ the following conditions are equivalent:
$$\bigl(\forall c\in\mathfrak{u}(\Hc_{\pm})\cap L_1(\Hc_{\pm})\bigr)\quad \Tr(bc)=0
\iff
\bigl(\forall c\in L_1(\Hc_{\pm})\bigr)\quad \Tr(bc)=0.
$$
Moreover the condition
$$\bigl(\forall c\in\mathfrak{u}(\Hc_{\pm})\cap L_1(\Hc_{\pm})\bigr)\quad
\Tr(bc)\in\mathbb{R}
$$
implies
$$\bigl(\forall c\in\mathfrak{u}(\Hc_{\pm})\cap L_1(\Hc_{\pm})\bigr)\quad
\Tr(b+b^{*})c=0,
$$
hence $b$ belongs to $\mathfrak{u}(\Hc_{\pm})$.
On the other hand, the
duality pairing of complex Banach spaces
$$
L_p(\Hc_{-},\Hc_{+})\times L_q(\Hc_{+},\Hc_{-})
\to{\mathbb C},\quad (b,c)\mapsto\Tr(bc),
$$
induces a duality pairing of the underlying real Banach spaces by
\begin{equation}\label{s2r}
L_p(\Hc_{-},\Hc_{+})\times L_q(\Hc_{+},\Hc_{-})\rightarrow
{\mathbb R},\quad (b,c)\mapsto\Re\,\Tr(bc).
\end{equation}
In view of formula \eqref{pairing}, we conclude that the restricted trace
induces a topological isomorphism of real Banach spaces
$$(\mathfrak{u}_{1, q})^*\cong\mathfrak{u}_{\res, p}.$$
\end{proof}

\section{Central extension of $\mathfrak{u}_{1,q}$ as Banach Lie-Poisson space}\label{central_extension}
Let us recall that a continuous $\mathbb{K}$-valued ($\mathbb K\in \{\mathbb R,\mathbb C\}$) 2-cocycle on a Lie algebra $\mathfrak{g}$ is a continuous bilinear map $\Phi: \mathfrak{g}\times\mathfrak{g}\rightarrow \mathbb{K}$ that is skew-symmetric 
\[
\Phi(A, B) = - \Phi(B, A)
\]
and satisfies the cocycle condition
\[
\Phi([A, B], C) + \Phi([C, A], B) + \Phi([B, C], A) = 0
\]
for all $A, B, C \in \mathfrak{g}$. In what follows, we restrict the Schwinger  cocycle (\cite{BRT07, segal-wilson,Wurzbacher}) to the Lie algebra $\mathfrak{u}_{\res, p}\subseteq \mathfrak u_{\mathrm{res},2}$ using the restricted trace.
\begin{definition}\label{def:schwinger}
    For $1< p\leq 2$, and $A, B\in L_{\res, p}(\Hc)$, define the Schwinger $2$-cocycle by 
    \begin{equation}\label{schwinger}
s(A, B) := \mathrm{Tr_{res}} (A[d, B]),
\end{equation}
\end{definition}
\begin{proposition}
    For $1< p\leq 2$, $s$ is a continuous
two-cocycle on $L_{\res, p}(\Hc)$.
\end{proposition}
\begin{proof}
Observe that $s$ is well-defined since $[d,B]$ is an block off-diagonal operator in $L_p(\mathcal H)\subseteq L_q(\mathcal H)$, hence is contained in $L_{1,q}(\mathcal H)$. Therefore, by Proposition~\ref{restricted trace properties}, it can be paired with $A\in L_{\mathrm{res},p}(\mathcal H)$. 
\begin{enumerate}
  \item Let us show that $s$ is skew-symmetric.
    Using the block decomposition of $A$ and $B \in L_{\res, p}(\Hc)$ with respect to the Hilbert sum $\Hc = \Hc_+\oplus \Hc_-$, one has
    \begin{equation*}
        \begin{split}
            s(A,B) &= \mathrm{Tr_{res}}(A[d,B])\\
            &= (-2i)\mathrm{Tr}(A_{+-}B_{-+}) +(2i)\mathrm{Tr}(A_{-+}B_{+-})\\
            &= (-2i)\mathrm{Tr}(B_{-+}A_{+-}) +(2i)\mathrm{Tr}(B_{+-}A_{-+}) = -\mathrm{Tr_{res}}(B[d,A]).
        \end{split}
    \end{equation*}
  \item Let us show that $s$ satisfies the cocycle identity. For $A,B,C\in L_{\mathrm{res},p}(\mathcal H)$ one has
  \begin{equation*}
       \begin{split}
                \sum\nolimits_{\mathrm{cycl}(A,B,C)} s([A,B],C) 
                =&\mathrm{Tr_{res}}([A,B][d,C]) +\mathrm{Tr_{res}}([C,A][d,B]) +\mathrm{Tr_{res}}([B,C][d,A])\\
                =& \mathrm{Tr_{res}}((AB-BA)[d,C]) + \mathrm{Tr_{res}}((CA-AC) [d,B]) - \mathrm{Tr_{res}}(A[d,[B,C]])\\
            =& \mathrm{Tr_{res}}(AB[d,C]) - \mathrm{Tr_{res}}(BA[d,C]) +\mathrm{Tr_{res}}(CA[d,B]) \\
            &- \mathrm{Tr_{res}}(AC[d,B]) - \mathrm{Tr_{res}}(A[d,[B,C]])
            \end{split}
  \end{equation*}
    where we have used that $AB, BA, CA,AC, BC$ and $CB$ are in $L_{\res, p}(\Hc),$ and $[d, A], [d, B]$ and $[d, C]$ are in $L_{1,q}(\Hc)$, as well as the skew-symmetry of $s$. It follows that 
    \begin{equation*}
            \begin{split}
                \sum\nolimits_{\mathrm{cycl}(A,B,C)} s([A,B],C)  =& \mathrm{Tr_{res}}(AB[d,C]) - \mathrm{Tr_{res}}(A[d,C]B) + \mathrm{Tr_{res}}(A[d,B]C) \\
                &- \mathrm{Tr_{res}}(AC[d,B]) - \mathrm{Tr_{res}}(A[d,[B,C]])\\
                =&\mathrm{Tr_{res}}(A[[C,d],B]) + \mathrm{Tr_{res}}(A[[d,B],C]) + \mathrm{Tr_{res}}(A[[B,C],d])\\
                =& \mathrm{Tr_{res}}(A([[C,d],B] +[[d,B],C] + [[B,C],d]))= 0
            \end{split}
        \end{equation*}
    since $[[C, d], B] + [[d, B],C] + [[B, C], d]$ belongs to $L_{1, q}(\Hc)$  and is equal to $0$ by the Jacobi identity in $\mathcal B(\Hc).$
    \item Let us show that $s$ is continuous as map from $L_{\res, p}(\Hc)\times L_{\res, p}(\Hc)$ into $\mathbb{C}.$ One has
    \[
    \begin{array}{ll}
    |s(A, B)|  & \leq |\Tr (A_{+-} B_{-+}) | +  |\Tr (A_{-+} B_{+-}) |\\
     & \leq\|A_{+-}\|_p\|B_{-+}\|_q + \|A_{-+}\|_p \|B_{+-}\|_q\\
     & \leq \|A_{+-}\|_p\|B_{-+}\|_p + \|A_{-+}\|_p \|B_{+-}\|_p \ \textrm{ $($since } L_p\subset L_q)\\ 
     & \leq 2 \|A\|_{\res,p}\|B\|_{\res, p} 
    \end{array}
    \]
\end{enumerate}

\end{proof}

\begin{definition}
We define the Banach Lie algebra $\tilde{\mathfrak{u}}_{\res, p}$
as the central extension of $\mathfrak{u}_{\res, p}$ with continuous
two-cocycle $s$  given by
\begin{equation}\label{schwinger definition}
s(A, B) := \mathrm{Tr_{res}} (A[d, B]),
\end{equation}
for all $A,B \in \mathfrak{u}_{\res, p}$. 
That is, $\tilde{\ug}_{\res, p}$ is
the Banach Lie algebra $\ug_{\res, p} \oplus \mathbb{R}$ endowed with the bracket
$[\cdot, \cdot]_{d}$ defined by
\begin{equation}\label{bracket u^d}
\left[(A, a),(B, b)\right]_{d} = \left([A, B], s(A, B)\right).
\end{equation}
We will denote by
$\langle\cdot\,, \cdot\rangle_{d}$ the duality pairing between
$\tilde{\ug}_{1, q} = \ug_{1, q} \oplus \mathbb{R}$ and
$\tilde{\ug}_{\res, p} = \ug_{\res, p} \oplus \mathbb{R}$ given by
$$
\left\langle (\mu, \gamma), (A, a) \right\rangle_{d} = \langle
\mu, A\rangle + \gamma a.
$$
\end{definition}

\begin{remark}
    {\rm Observe that the continuous $2$-cocycle \eqref{schwinger definition} takes values in $\mathbb R$. Indeed, for $A,B\in \mathfrak u_{\mathrm{res},p}$ one has}
    \begin{equation*}
      \begin{split}
          \overline{s(A,B)}&= \overline{\mathrm{Tr}_{\mathrm{res}}(A[d,B])}\\ &= \mathrm{Tr}_{\mathrm{res}}([d,B]^*A^* ) \\&=\mathrm{Tr}_{\mathrm{res}}([B,d](-A))\\
          &=\mathrm{Tr_{res}}(A[d,B]) = s(A,B).
      \end{split}
    \end{equation*}
    
\end{remark}

\begin{remark}\label{Lie algebra ures,p}
    {\rm Continuity of the bracket $[\cdot,\cdot]_d$ follows from continuity of the bracket on $\mathfrak u_{\mathrm{res},p}$ and continuity of the $2$-cocycle. Moreover, the $2$-cocycle identity of $s$ implies that $[\cdot,\cdot]_d$ satisfies the Jacobi identity.}
\end{remark}

\begin{remark}
    {\rm By Proposition~\ref{duality}, $\tilde{\ug}_{1, q}$ is a predual of $\tilde{\ug}_{\res, p}$
    \[
(\tilde{\ug}_{1, q})^* = \tilde{\ug}_{\res, p}   
    \]
    }
\end{remark}
The following Proposition is an extension of Proposition~2.5 in \cite{BRT07}
to the case $1< p\leq 2$, see also Theorem~3.14 in \cite{T19} for arbitrary duality pairings.
\begin{proposition}\label{Banach_Poisson}
Set $1< p\leq 2$ and $q$ such that $\frac{1}{p}+\frac{1}{q} = 1$, and $q = \infty$ for $p = 1$. The Banach space $\tilde{\ug}_{1, q}:= \ug_{1, q}\oplus\mathbb{R}$ is a Banach
Lie-Poisson space with respect to $\tilde{\ug}_{\res, p}$ for the Poisson bracket
\begin{equation}\label{bracket_d2}
\{f,g\}_{d}(\mu, \gamma) := \langle \mu,  \left[ D_{\mu}f(\mu,\gamma),
D_{\mu}g(\mu,\gamma) \right] \rangle + \gamma s(D_{\mu}f(\mu,\gamma), D_{\mu}g(\mu,\gamma))
\end{equation}
where $f, g$ are smooth real functions on $\tilde{\ug}_{1, q}$, $(\mu,
\gamma)$ is an arbitrary element in $\tilde{\ug}_{1, q}$,
and $D_{\mu}$ denotes the partial Fr\'echet derivative with
respect to $\mu \in \tilde{\ug}_{1, q}$.
\end{proposition}

\begin{proof}
    Since $\tilde{\ug}_{1, q}$ is a predual of $\tilde{\ug}_{\res, p}$, by Theorem~4.2 in \cite{OR03}, the Banach space
$\tilde{\ug}_{1, q}$ is a Banach Lie-Poisson space with respect to $\tilde{\ug}_{\res, p}$ if and
only if $\tilde{\ug}_{\res, p}$ is a Banach Lie algebra
satisfying $\ad_{A}^{*}(\tilde{\ug}_{1, q}) \subset
\tilde{\ug}_{1, q} \subset (\tilde{\ug}_{\res, p})^{*}$ for all
$A \in \tilde{\ug}_{\res, p}$. By Remark \ref{Lie algebra ures,p}, $\mathfrak u_{\mathrm{res},p}$ is a Lie algebra. To see that the coadjoint action
of $\tilde{\ug}_{\res, p}$ preserves the predual
$\tilde{\ug}_{1, q}$, note that for every $(A, a),(B, b) \in
\tilde{\ug}_{\res, p}$ and every $(\mu, \gamma) \in
\tilde{\ug}_{1, q}$, one has
\begin{equation}\label{computation_coad}
\begin{array}{ll}
\langle -\ad_{(A, a)}^{*}(\mu, \gamma), (B, b)\rangle_{d} &:= 
  \langle (\mu, \gamma), -\ad_{(A, a)}(B, b)\rangle_{d}\\&
 =  \langle (\mu,
\gamma), -[(A, a), (B, b)]_{d}\rangle_{d}\\
&=  \langle (\mu, \gamma), \left(-[A, B], - s(A, B)\right)
\rangle_{d}\\&
 = -\mathrm{Tr_{res}} \mu [A, B] - \gamma \mathrm{Tr_{res}} A[d, B]\\
&=  -\mathrm{Tr_{res}} \mu [A, B]- \gamma \mathrm{Tr_{res}} [A, d] B  \\& =  \langle
(-\ad^{*}_A(\mu) - \gamma [A, d], 0), (B, b) \rangle_{d}.
\end{array}
\end{equation}
It follows from the non-degeneracy of the pairing $\langle \cdot, \cdot\rangle_{d}$, that 
\[
-\ad_{(A, a)}^{*}(\mu, \gamma) = (-\ad^{*}_A(\mu) - \gamma [A, d], 0),
\]
where $(A, a) \in \tilde{\ug}_{\res, p}$ and $(\mu, \gamma)\in \tilde{\ug}_{1, q}$.
Note that $\gamma [A, d]\in L_{1,q}(\Hc)$ for $1< p\leq 2$, and is skew-adjoint. Moreover,
\[
\begin{split}
    \langle -\ad^{*}_A(\mu), B\rangle = -\mathrm{Tr_{res}} (\mu [A, B]) &= -\mathrm{Tr_{res}}(\mu AB - \mu BA) \\&= -\mathrm{Tr_{res}}(\mu AB) +\mathrm{Tr_{res}}(\mu BA),    
\end{split}
\]
since $\mu AB$ and $\mu BA$ are in $L_{1, q}(\Hc)$. By Proposition~\ref{restricted trace properties}, 
\[
\mathrm{Tr}_{\res}(\mu BA) = \mathrm{Tr}_{\res}(A\mu B),
\]
since $\mu B\in L_{1, q}(\Hc)$. Therefore
\[
-\mathrm{Tr}_{\res}(\mu AB) +\mathrm{Tr}_{\res}(\mu BA) = -\mathrm{Tr}_{\res}([\mu, A]B),
\]
which implies that $-\ad^{*}_A(\mu) = -[\mu, A]$ and
\[
-\ad_{(A, a)}^{*}(\mu, \gamma) = ([A, \mu] - \gamma [A, d], 0) = ([A, \mu -\gamma d], 0).
\]
 Since $L_{1, q}(\Hc)\cdot L_{\res, p}(\Hc) \subset L_{1, q}(\Hc)$, it follows that the coadjoint action
of $\tilde{\ug}_{\res, p}$ preserves the predual
$\tilde{\ug}_{1, q}$.
The expression of the Poisson bracket is standard.
\end{proof}

\begin{remark}
    {\rm For each $\gamma\in \mathbb R$, the affine subspace $\mathfrak u_{1,q}\oplus \{\gamma\}$ is a Banach Poisson manifold with respect to the Poisson structure induced from $\mathfrak u_{1,q}\oplus \mathbb R$. }
\end{remark}

\section{The restricted $p$-Schatten class Grassmannian as affine coadjoint orbit of $\operatorname{U}_{\res, p}(\mathcal{H})$}

By Proposition~\ref{restricted_trace_conjugation_invariance}, the coadjoint action of $\operatorname{U}_{\res, p}(\Hc)$  on 
$\ug_{1, q} $ is given by
\[
\begin{array}{ll}
\langle \Ad^*_{g^{-1}}(\alpha), A\rangle & := \langle \alpha, \Ad_{g^{-1}}(A)\rangle \\
& = \mathrm{Tr_{res}}(\alpha g^{-1} A g) = \mathrm{Tr_{res}}(g\alpha g^{-1} A)\\
&= \langle g\alpha g^{-1}, A\rangle,
\end{array}
\]
where $\alpha \in \ug_{1, q}$ and $A \in \mathfrak{u}_{\res, p}$. By the non-degeneracy of the pairing $\langle\cdot, \cdot\rangle$ given in Proposition~\ref{duality}, it follows that the coadjoint action of $\operatorname{U}_{\res, p}(\Hc)$  on 
$\ug_{1, q} $ reads
\[
\Ad^*_{g^{-1}}(\alpha) = g\alpha g^{-1}.
\]
In other words, the invariance of the pairing \eqref{conj} implies that the coadjoint action of $\operatorname{U}_{\res, p}(\Hc)$  on 
$\ug_{1, q} $ reduces to the action by conjugation.  

\begin{proposition}\label{affine_coadjoint}
The  unitary group $
\operatorname{U}_{\res, p}(\Hc)$ acts on the Poisson manifold
$\ug_{1, q} \oplus \{\gamma\} \subset
\tilde{\ug}_{1, q}$ by the affine coadjoint action as follows.
For  $g \in \operatorname{U}_{\res, p}(\Hc)$,
\begin{equation*}
g\cdot(\mu, \gamma) := \left(\Ad^{*}_{g^{-1}}(\mu) 
-\gamma\sigma(g), \gamma\right) 
\end{equation*}
where $\mu \in \ug_{1, q}$, $\gamma\in\mathbb{R}$, and where
\begin{equation*}
    \sigma\colon \operatorname{U}_{\res, p}(\Hc) \to \ug_{1, q},\quad g  \mapsto  gdg^{-1} - d.
\end{equation*}
\end{proposition}

\begin{proof}
Let us verify that for every $g\in \mathrm{U}_{\mathrm{res},p}(\mathcal H)$ we have $gdg^{-1} -d \in \mathfrak u_{1,q}$. Observe that 
\begin{equation*}
    \sigma(g) = gdg^{*} -d = [g,d]g^*.
\end{equation*}
Consider the block decomposition of $g$ with respect to the Hilbert sum $\mathcal H = \mathcal H_+\oplus \mathcal H_-$
\begin{equation*}
    g =\begin{pmatrix} g_{++} & g_{+-} \\
                    g_{-+} & g_{--}
    \end{pmatrix}\in\operatorname{U}_{\res, p}(\Hc).
\end{equation*}
Then 
\begin{equation*}
    [g,d]g^* = 2i\begin{pmatrix}
        0 & -g_{+-}\\ g_{-+} & 0
    \end{pmatrix}\begin{pmatrix}
        g^*_{++} & g^*_{-+} \\ g^*_{+-} & g^*_{--}
    \end{pmatrix} = 2i\begin{pmatrix}
        -g_{+-}g^*_{+-} & -g_{+-}g^*_{--}\\ g_{-+}g^*_{++} & g_{-+}g^*_{-+}
    \end{pmatrix}.
\end{equation*}
The off-diagonal blocks are of class $L_q$ as their product lands in $L_p$ which, by the condition on $p$, is contained in $L_q$. The diagonal blocks are of trace class as $1< p \leq 2$, we have that  $L_p\cdot L_p \subseteq L_p\cdot L_q \subseteq L_1$. On the other hand we have 
\begin{equation}\label{product}
gdg^*
    =
\begin{pmatrix}
i g_{++}g_{++}^{*}-i g_{+-}g_{+-}^{*}
      & i g_{++}g_{-+}^{*}-i g_{+-}g_{--}^{*} \\
i g_{-+}g_{++}^{*}-i g_{--}g_{+-}^{*}
      & i g_{-+}g_{-+}^{*}-i g_{--}g_{--}^{*}.
    \end{pmatrix}.
\end{equation}
This multiplication is clearly skew-Hermitian, since $d$ is also skew-Hermitian we conclude that $\sigma(g)\in \mathfrak u_{1,q}$ for all $g\in \mathrm{U}_{\mathrm{res},p}(\mathcal H)$.

Denoting by
$\textrm{Aff}\left(\ug_{1,q}\oplus\{\gamma\}\right)$ the
affine group of transformations of
$\ug_{1,q}\oplus\{\gamma\}$, it remains to show that
$$
\begin{aligned}
(\Ad^*,-\gamma\sigma)\colon \operatorname{U}_{\res, p}(\Hc) &\to
\textrm{Aff}(\ug_{1,q}\oplus\{\gamma\}) =
\textrm{GL}(\ug_{1,q}\oplus\{\gamma\})
\rtimes \ug_{1,q}\\
g &\mapsto(\Ad^{*}_{g^{-1}}, -\gamma\sigma(g))
\end{aligned}
$$
is a group homomorphism. For this, we have to check that 
\begin{equation*}
    \sigma(g_{1}g_{2}) = \Ad^{*}_{g_{1}^{-1}}\sigma(g_{2}) +
\sigma(g_{1})
\end{equation*}
for all $g_{1}$, $g_{2}$ in $\operatorname{U}_{\res, p}(\Hc)$. This is the same computation as in the proof of Proposition 2.9 in \cite{BRT07} and we omit it here. 
\end{proof}

\begin{proposition}\label{isotropy}
For $\gamma\neq 0$, the isotropy group of $(0,
\gamma)\in(\ug_{\res,p})_{*}\oplus\{\gamma\}$ for the
$\operatorname{U}_{\res, p}(\Hc)$-affine coadjoint action  is the Banach Lie subgroup $\operatorname{U}(\Hc_+)\times \operatorname{U}(\Hc_-)$ of
$\operatorname{U}_{\res, p}(\Hc)$.
\end{proposition}

\begin{proof}
 For
$\mu = 0$ and $\gamma\neq0$, the stabilizer of $(0, \gamma)$
consists of all elements of $\operatorname{U}_{\res, p}(\Hc)$ which commute with $d$.
Hence, for $\gamma\neq0$, the Lie algebra of the isotropy is
$$\ug_{(0, \gamma)}:=\ug(\Hc_{+})\oplus\ug(\Hc_{-}).$$
A topological complement to
$\ug_{(0, \gamma)}$ in $\ug_{\res, p}$ can be chosen as
$$\mathfrak{m}_p:=\ug(\Hc)\cap\left(L_p(\Hc_{+},
\Hc_{-})\oplus L_p(\Hc_{-},\Hc_{+})\right).$$
By a classical result in \cite{bourbaki-v}, $\operatorname{U}(\Hc_+)\times \operatorname{U}(\Hc_-)$ is a Banach Lie subgroup of
$\operatorname{U}_{\res, p}(\Hc)$.
\end{proof}

\begin{theorem}
The restricted $p$-Schatten class Grassmannian $\mathrm{Gr}_{\mathrm{res},p}(\mathcal H)$ is diffeomorphic to the affine coadjoint orbit $\mathcal{O}_{(0, \gamma)}$ of the group $\operatorname{U}_{\res, p}(\Hc)$ passing through any $(0, \gamma)\in \mathfrak{u}_{1, q}$ with $\gamma\neq0$.
\end{theorem}
\begin{proof}
  Observe that 
\[
\begin{array}{ll}
     \mathcal{O}_{(0, \gamma)} & =\left\{ (-\gamma \sigma(g), \gamma)\mid g\in \operatorname{U}_{\res, p}(\Hc)\right\}\\ & = \left\{ (-\gamma g d g^{-1} - \gamma d, \gamma)\mid g\in \operatorname{U}_{\res, p}(\Hc)\right\}.
\end{array}
\]
    By Proposition~\ref{isotropy}, for $\gamma\neq 0$,  the affine coadjoint orbit of $(0, \gamma)\in \mathfrak{u}_{1, q}$ is diffeomorphic to the quotient space $\operatorname{U}_{\res, p}(\Hc)/\left(\operatorname{U}(\Hc_+)\times \operatorname{U}(\Hc_-)\right)$. By Proposition~\ref{homogeneous}, this quotient space is diffeomorphic to the restricted $p$-Schatten class Grassmannian $\mathrm{Gr}_{\mathrm{res},p}(\mathcal H)$. An identification of $\mathrm{Gr}_{\mathrm{res},p}(\mathcal H)$ with $\mathcal O_{(0,\gamma)}$ is given by 
    \begin{equation*}
        \Phi_\gamma\colon \mathrm{Gr}_{\res,p}(\mathcal H)\to \mathcal O_{(0,\gamma)},\quad W\mapsto \gamma(i (\mathrm{pr}_W-\mathrm{pr}_{W^\perp}) - i(\mathrm{pr}_+ - \mathrm{pr}_-)).
    \end{equation*}
    This map is well-defined due to Propositions \ref{projection grassmannian} and \ref{dual_Grassmannian}. 
\end{proof}

\section{The restricted Grassmannian $\mathrm{Gr}_{\res, p}(\mathcal{H})$, $1< p\leq 2,$ as  symplectic manifold}

\subsection{Homogeneous symplectic form on  $\mathrm{Gr}_{\mathrm{res},p}(\mathcal H)$}

\begin{theorem}\label{homogeneous form}
    For $1< p\leq 2$, the homogeneous space 
    $$\mathrm{Gr}_{\mathrm{res},p}(\mathcal H) = \operatorname{U}_{\res, p}(\Hc)/\left(\operatorname{U}(\Hc_+)\times \operatorname{U}(\Hc_-)\right)$$ admits a natural $\operatorname{U}_{\res, p}(\Hc)$-invariant weak symplectic structure, whose expression at the class of the identity operator is given by
    \begin{equation}
        \Omega_{[\mathrm{id}]}\left([A], [B]\right) =   2\Im \Tr(A^*_{-+}B_{-+}) = -\frac{1}{2} \mathrm{Tr_{res}} A [d, B] = -\frac{1}{2} s(A, B),
    \end{equation}
    where 
    $$[A] = \left[\begin{pmatrix}A_{++} & -A^*_{-+}\\A_{-+} & A_{--}  \end{pmatrix}\right] = \left[\begin{pmatrix}0 & -A^*_{-+}\\A_{-+} & 0  \end{pmatrix}\right]$$ denotes the class of $A$ modulo the isotropy Lie algebra, and similarly for $B$.

\end{theorem}
\begin{proof}
    The form $\Omega_{[\mathrm{id}]}$ is continuous due to continuity of the inner-product on $L_2$ and the continuous inclusions $L_p\hookrightarrow L_2$. 
    By the $\operatorname{U}_{\res, p}(\Hc)$-invariance, the symplectic structure is uniquely determined by its value at the tangent space at one point. The tangent space at the class of the identity operator (modulo the isotropy) is isomorphic to $L_p(\Hc_+, \Hc_-)$ and is included in $L_2(\Hc_+, \Hc_-)$ by the condition on $p$. The symplectic form coincides with twice the $L_2$-scalar product on $L_p(\Hc_+, \Hc_-)\subset L_2(\Hc_+, \Hc_-)$.
Since $\Omega_{[\id]}$ is invariant by the isotropy, it defines a $2$-form on the quotient space, which is clearly non-degenerate, but weak. The fact that $\Omega$ is closed follows from the cocycle condition satisfied by the Schwinger term (\cite[Lemma VI.1]{Nee1}). Denote by $L_P$ the natural left action of $g\in \mathrm U_{\mathrm{res},p}(\mathcal H)$ on $\mathrm{Gr}_{\mathrm{res},p}(\mathcal H)$. The symplectic form $\Omega$ at the class of the operator $g$ is given by 
\begin{equation*}
    \Omega_{[g]}([A],[B]) = \Omega_{[\mathrm{id}]}((L_{g^{-1}})_*[A],  (L_{g^{-1}})_*[B])
\end{equation*}
for $[A],[B]\in T_{[g]}\mathrm{Gr}_{\res,p}(\mathcal H)$, where $L_{g^{-1}}$ is the left translation by $g^{-1}$.
\end{proof}

\subsection{Kirillov-Kostant-Souriau symplectic form on the affine coadjoint orbits $\mathcal{O}_{(0, \gamma)}$}

\begin{theorem}\label{KKS form}
    For $1\leq p\leq 2$ and $\gamma\neq 0$, the symplectic form induced by the Lie-Poisson space $\tilde{\mathfrak u}_{1, q}$ on the affine coadjoint orbit $\mathcal{O}_{(0, \gamma)}$ is the $\operatorname{U}_{\res, p}(\Hc)$-invariant weak symplectic structure whose expression at the tangent space to $(0, \gamma)$ is given by
    \begin{equation}
        \omega_{(0, \gamma)}\left(X^{(A, a)}, X^{(B, b)}\right)  = \gamma s(A, B),
    \end{equation}
    where $X^{(A, a)} = -\ad^*_{(A, a)}(0, \gamma)$ and $X^{(B, b)} = -\ad^*_{(B, b)}(0, \gamma)$.
    \end{theorem}

    \begin{proof}
    The Lie-Poisson structure of $\tilde{\mathfrak u}_{1, q}$ introduced in Section~\ref{central_extension} induces the $\operatorname{U}_{\res, p}(\Hc)$-invariant symplectic structure on the affine coadjoint orbit $\mathcal{O}_{(0, \gamma)}$ whose expression at the tangent space to $(0, \gamma)$ is
    \[
    \begin{array}{ll}
        \omega_{(0, \gamma)}\left(X^{(A, a)}, X^{(B, b)}\right) & =  \langle(0, \gamma), [(A, a), (B, b)]_d\rangle_d 
        \\ & =  
        \langle(0, \gamma), ([A, B], s(A, B))\rangle_d \\&  = \gamma s(A, B),
    \end{array}
    \]
    where $X^{(A, a)} = -\ad^*_{(A, a)}(0, \gamma) = ([A, -\gamma d], 0)$ and $X^{(B, b)} = -\ad^*_{(B, b)}(0, \gamma) = ([B, -\gamma d], 0)$.
    \end{proof}

    \begin{corollary}
        The restricted Grassmannian $\mathrm{Gr}_{\res, p}(\mathcal{H})$, $1\leq  p\leq 2,$ admits a one parameter family of symplectic forms, which coincide with the Kirillov-Kostant-Souriou symplectic forms of $\mathcal{O}_{(0, \gamma)}$ after the identification
        \[
        \begin{array}{lll}
            \mathrm{Gr}_{\res,p}(\mathcal{H}) & \longrightarrow & \mathcal{O}_{(0, \gamma)}\\
             W & \longmapsto & \gamma \left(i(\mathrm{pr}_{W} - \mathrm{pr}_{{W}^\perp}) - i(\mathrm {pr}_+ - \mathrm{pr}_-)\right).
        \end{array}
        \]
        where $\mathrm {pr}_{W}$ $($resp. $\mathrm{pr}_{{W}^\perp})$ is the orthogonal projection onto $W$ $($resp. $W^\perp)$.
    \end{corollary}

    \begin{proof}
        This is a consequence of definition of the restricted Grassmannian $\mathrm {Gr}_{\res, p}(\mathcal{H})$ using orthogonal projections.
    \end{proof}

\begin{remark}
        {\rm The homogeneous symplectic form defined in Theorem \ref{homogeneous form} is proportional to the KKS-form from Theorem \ref{KKS form} after the identification of the restricted $p$-Schatten class Grassmannian with the affine coadjoint orbit.} 
\end{remark}

\end{document}